\documentclass[12pt]{amsart}
\usepackage{amssymb}
\usepackage{amsfonts}
\usepackage{amsmath,amsthm,amsfonts,amscd,amssymb,mathdots,amstext}
\usepackage{fullpage}
\usepackage[all]{xy}
 \usepackage[usenames,dvipsnames]{color}

\usepackage{graphicx}
\usepackage{bbding}
\usepackage[misc]{ifsym}

\usepackage[margin=3.65cm]{geometry}

\theoremstyle{plain}
\newtheorem{thm}{Theorem}[section]
\newtheorem{cor}[thm]{Corollary}
\newtheorem{prop}[thm]{Proposition}
\newtheorem{lem}[thm]{Lemma}

\newcounter{theoremintro}
\newtheorem{theoremi}[theoremintro]{Theorem}

\theoremstyle{definition}

\newtheorem{defn}[thm]{Definition}

\newtheorem{eg}[thm]{Example}

\newcommand{\bC}{{\mathbb{C}}}
\newcommand{\bD}{{\mathbb{D}}}

\newcommand{\bN}{{\mathbb{N}}}

\newcommand{\bZ}{{\mathbb{Z}}}

  \newcommand{\A}{{\mathcal{A}}}

\renewcommand{\H}{{\mathcal{H}}}

\renewcommand{\L}{{\mathcal{L}}}
  \newcommand{\M}{{\mathcal{M}}}
  \newcommand{\N}{{\mathcal{N}}}

  \newcommand{\X}{{\mathcal{X}}}

\renewcommand{\phi}{\varphi}
\newcommand{\upchi}{{\raise.35ex\hbox{$\chi$}}}

\newcommand{\ran}{\operatorname{ran}}

\begin{document}

\title[The invariant subspace problem and Rosenblum operators]{The invariant subspace problem and Rosenblum operators I}

\author{Junsheng Fang}
\address{Junsheng Fang, School of Mathematics, Hebei Normal University, 050016, Shijiazhuang, P. R. China}
\email{jfang@hebtu.edu.cn}

\author{Bingzhe Hou}
\address{Bingzhe Hou, School of Mathematics, Jilin University, 130012, Changchun, P. R. China}
\email{houbz@jlu.edu.cn}
	
\author{Chunlan Jiang}
\address{Chunlan Jiang, School of Mathematics, Hebei Normal University, 050016, Shijiazhuang, P. R. China}
\email{cljiang@hebtu.edu.cn}

\author{Yuanhang Zhang}
\address{Yuanhang Zhang, School of Mathematics, Jilin University, 130012, Changchun, P. R. China}
\email{zhangyuanhang@jlu.edu.cn}

\thanks{Junsheng Fang is supported in part by National Natural Science Foundation of China (No.: 12071109); Chunlan Jiang and Bingzhe Hou are supported in part by National Natural Science Foundation of China (No.:12471120); Chunlan Jiang and Junsheng
Fang are partly supported by the Hebei Natural Science Foundation (No.: A2023205045); Yuanhang Zhang is supported in part by National Natural Science Foundation of China (No.: 12471123).
This article was also partly supported by Tianyuan Mathematics Research Center.}
	
\date{}
\thanks{}
\begin{abstract}
Let $T\in B(\H)$ be an invertible operator. From the 1940's, Gelfand, Hille and Wermer investigated the invariant subspaces of $T$
by analyzing the growth of $\|T^n\|$, where $n\in \bZ$. In this paper, we study the invariant subspaces of $T$ by estimating the
growth of $\|T^n+\lambda T^{-n}\|$, where $n\in \bN$ and $\lambda$ is a nonzero complex constant. The key ingredient of our approach is
introducing the notion of shift representation operators,
which is based on the Rosenblum operators. In addition, by employing shift representation operators, we provide an equivalent
of the Invariant Subspace Problem via the injectivity of certain Hankel operators.
\end{abstract}

\date{\today}

\maketitle

Let $\H$ denote the unique (up to isomorphism)
separable, infinite-dimensional complex Hilbert space.
A {\it subspace} of $\H$ is a linear manifold which is closed in the norm topology.
If $\L\subset \H$, then the {\it span} of $\L$, denoted by $\bigvee\L$
is the intersection of all subspaces containing $\L$.
We say that the subspace $\M$
is invariant under the operator $T$
if $T\M\subset \M$. Denote by $B(\H)$ the set of all
bounded linear operators acting on $\H$. For $T\in B(\H)$, the collection of all subspaces of $\H$ invariant under $T$ is denoted $\textup{Lat}T$; $T$ is called {\it intransitive} if it leaves invariant
subspaces other than $\{0\}$ or $\H$; otherwise, it is called
{\it transitive} which means that $\textup{Lat}T=\{\{0\},\H\}$.
Denote by $K(\H)$ the closed ideal of $B(\H)$ of all compact operators. For $T\in B(\H)$, denote the kernel of
$T$ and the range of $T$ by $\ker T$ and $\ran T$ respectively; and denote by $\sigma(T),
\sigma_p(T), r(T)$ the spectrum, the point
spectrum, the spectral radius of $T$ respectively. Similarly, the above notations are also used in the Banach space setting.

Given vectors $f$ and $g\in \H$, we define the rank one operator $f\otimes g$
mapping $\H$ into itself by $(f\otimes g)h=\langle h, g\rangle f$, $\forall h\in \H$.
To each $T\in B(\H)$, we could naturally associate a weakly closed subalgebra with identity of $B(\H)$,
\[\{T\}'=\{A\in B(\H):TA=AT\}=\textup{the}~\it{commutant}~\textup{of}~T.\]

The \textbf{Invariant Subspace Problem} is one of best-known unsolved problems in Functional Analysis:

{\it Does every bounded operator on a separable, infinite dimensional, complex Hilbert space have a non-trivial invariant subspace?}

In the 1930's von Neumann found a very beautiful technique for producing invariant subspaces for
compact operators by approximating by certain invariant
subspaces of a sequence of finite-rank operators. This result
was never published. In 1954, Aronszajn and Smith extended the result to separable Banach spaces \cite{AS}.
Furthermore, Lomonosov \cite{Lom} greatly increased the class of operators with proper invariant subspaces by showing that every operator which commutes with a compact operator has a proper invariant subspace.

In 1975, Enflo presented the first construction of an operator on a Banach space having no invariant subspace. The paper has existed in manuscript form for about 12 years prior to its publication \cite{En,Bea}.  In the 1980's, Read produced more examples of operators which have no proper invariant subspace on some non-reflexive Banach spaces \cite{R1,R2,R3,R4}. Dropping the boundedness assumption,  bijective linear operators having no invariant subspace were constructed on any separable infinite-dimensional Banach space \cite{Shi70,Had75}.

For an overview of the Invariant Subspace Problem, we refer to the classical monograph by Radjavi and Rosenthal \cite{RR2} and a more recent book by Chalendar and Partington \cite{CP} as well as Brown's remarkable papers \cite{Brown78,Brown87}.


A very important special case for which the Invariant Subspace Problem is still open is that of quasinilpotent operator on $\H$.  An operator $A$ in $B(\H)$ is called quasinilpotent if the spectrum of $A$, $\sigma(A)=\{0\}$. The reason why we say the quasinilpotent operators are much more important is the following result: if every operator in norm closure of quasinilpotent operators has a proper invariant subspace, then the answer is``Yes" for the Invariant Subspace Problem. If the readers want to consult more detailed results, one might refer to the following list: Apostol and Voiculescu \cite{ApoVoi74}, Herrero \cite{Her78}, Foias and Pearcy \cite{FoiPea74}; Foias, Jung, Ko, and Pearcy \cite{FJKP1,FJKP2}, Tcaciuc \cite{Tca19}.

Let $\X$ be a complex Banach space, $A$ be a bounded linear operator whose spectrum $\sigma(A)$ is
the singleton $\{1\}$, and $r$ be a positive integer. The Gelfand-Hille theorem asserts that $(A-I)^r=0$ if $\|A^n\|=O(|n|^{r-1})$ as $|n|\to \infty$ \cite{Gel41,Hil44}.
In 1952, Wermer \cite{Wer52} used this theorem to prove
that an invertible operator $T\in B(\X)$ has a nontrivial invariant subspace if $\|T^n\|=O(|n|^r)$
as $|n|\to \infty$ for some positive integer $r$.
So far, the Gelfand-Hille theorem has been generalized to some different versions, we refer
to the survey of Zem\'{a}nek \cite{Zem94}, \cite{DriZem00}, \cite{FHJ24} and the references therein.

Fix $A,B\in B(\H)$, recall that the Rosenblum operator: $\tau_{A,B}(X)=AX-XB$, $X\in B(\H)$.
 The operator $\tau_{A,B}$ was first systematically studied by Rosenblum \cite{Ros}. Later, many scholars participated in the research and got abundant results. For more details, please refer to the following literatures \cite{DavRos74,Fia78,Fia79,BhaRos97}.

 The unilateral shift operator $S$ acting on classical Hardy space $H^2$ has fruitful nice properties, for example, the celebrated Beurling Theorem \cite{Beu48}.
In this paper, we shall build a relation between $T$ and $S$ by considering the Rosenblum operator $\tau_{T,S}$.
 More precisely, we show that $x\mapsto K_x=\sum_{n=0}^\infty T^nx\otimes e_n$ is an one to one map from $H^2$ onto
$\ker \tau_{T,S}$, where $T\in B(H^2)$ with $r(T)<1$, and $\{e_n\}_{n=0}^\infty$ is an orthogonal normal basis of $H^2$.
We name such operator $K_x$ as a \textbf{shift representation operator}. The discovery of $K_x$ reveals some hidden analytic structure and geometric structure of $T$. After a finer and detailed analysis of these structures,  we could discuss the existence of nontrivial invariant subspaces of an invertible operator $T\in B(\H)$ by estimating the
growth of $\|T^n+\lambda T^{-n}\|$, where $n\in \bN$, $\lambda$ is a nonzero constant. By contrary with Gelfand, Hille, Wermer, the growth condition appears to be weakened in an asymmetric way. In a precise sense, we show that

%

\begin{theoremi}\label{A}
Let $T\in B(\H)$. If there exists an invertible operator $A\in \{T\}'\setminus \mathbb{C}I$, $r\in \bN$ and $\lambda\neq 0$
such that
\[\|A^n+\lambda A^{-n}\|\sim O(n^r),~~\forall n\in \bN,\]
then $T$ is intransitive.
\end{theoremi}

\begin{theoremi}\label{B}
Let $T\in B(\H)$. Suppose that $K\in K(\H)$, and $r\in \bN$, $\lambda\neq 0$, such that $\sigma(T+K)=\{0\}$,
and
\[\|(T+K+I)^n+\lambda (T+K+I)^{-n}\|\sim O(n^r),~~\forall n\in \bN.\]
Then $(T+K)$ is a nilpotent operator of order $r+4$. Hence $T$ is intransitive.
\end{theoremi}

In addition, by employing the shift representation operators, we provide an equivalent
of the Invariant Subspace Problem via the injectivity of certain Hankel operators.

\begin{theoremi}\label{C}
Let $A\in B(H^2)$ with $r(A)<1$. Then $A$ is intransitive if and only if there are two nonzero vectors $f,g\in H^2$, such that
the Hankel operator $H_\psi$ is not injective, where $\psi=\sum_{n=0}^{\infty}\langle A^n f,g\rangle z^{-n}$.
\end{theoremi}

Let $\X$ be a complex Banach space, $T\in B(\X)$ and $x\in \X$, then $\bigvee_{n=0}^{\infty}\{T^nx\}$
is invariant under $T$. The vector $x$ is a {\it cyclic} vector of $T$
if $\bigvee_{n=0}^{\infty}\{T^nx\}=\X$;
in this case, $T$ is called
a cyclic operator. It is evident that $T\in B(\X)$ is transitive if and only if every nonzero vector $x\in \X$ is a cyclic vector of $T$.
In this sense, the interest in studying cyclic operators stems from the
Invariant Subspace Problem.
In 1968, Sz-Nagy and Foias \cite{NF68} showed that the vector space spanned by the set of all cyclic vectors for a cyclic operator on a Hilbert space is dense in the whole space.
In 1972, Geher \cite{Geh72} gave an elementary proof of this fact for a cyclic operator on a Banach space. In 1995, Ansari \cite{Ans95} proved that for a cyclic operator $T$ on a Banach space the set of all cyclic vectors for $T$ is itself dense if $\sigma_p(T^*)$ contains no nonempty open set.

Finally, with the help of the new tool of shift representation operators, we have the following theorem, which is a surprise to the authors.

\begin{theoremi}\label{D}
Let $T\in B(H^2)$ with $\sigma(T)=\{0\}$ and $\sigma_p(T)=\emptyset$.
Suppose that $T^mx$ is a cyclic vector of $T$, $\forall m\in \bN$.
Then for any $0\neq\alpha=\{\alpha_k\}_{k=0}^{\infty}\in l^2$, $\sum_{k=0}^{\infty}\alpha_kT^kx$ is
a cyclic vector of $T$.
\end{theoremi}

\noindent \textbf{Structure of the paper.} In Section 1 we cover some preliminaries on Hardy spaces, introduce the notion of shift representation operators and study the properties. The proofs of
Theorem \ref{A} and Theorem \ref{B} are in Section 2.

In Section 3 we quickly provide some background on Hankel operators,
and then carry out the properties of the shift representation operators (Proposition \ref{Rosenblum 1} that Proposition \ref{two conditions make A intransitive}) to prove Theorem \ref{C}, which is an equivalent of the Invariant Subspace Problem via the injectivity of certain Hankel operators.
In Section 4 we prove Theorem \ref{D} and conclude the paper by giving a partial answer to a problem of Halmos.


\section{Rosenblum operators and shift representation operators}

In this section, we will explore the finer structure of $K$ in $\ker\tau_{A, S}$ to describe properties of the invariant subspaces of an operator $A$. This plays a key role in the paper. We first review briefly some of the facts about Hardy spaces which will be needed later.

Let $\mathbb{T}$ be the unit circle in the complex plane with the normalized Lebesgue measure $m$.
For $p=2, \infty$,
let $H^p$ be the classical Hardy space of all functions in $L^p(\mathbb{T})$ that have analytic extensions to
the open disk $\mathbb{D}$. Throughout the paper, we always use $\{e_n\}_{n=0}^\infty$ to denote the canonical orthonormal basis $\{z^n\}_{n=0}^\infty$ for $H^2$.
Let $S$ be the {\it unilateral shift of multiplicity 1}, which is the unique operator $S$ such that $Se_n=e_{n+1}$ for $n=0,1,2,\cdots$.
Up to Hilbert space isomorphism, $\H$, $H^2$, $l^2$ are the same, there is no harm in letting $S$ also denote the unilateral shift operator in $B(\H)$ or $B(l^2)$.

\begin{defn}
A measurable function $\phi$ on $\mathbb{T}$ is inner if $\phi\in H^2$ and $|\phi|=1$ a.e.. A function in $H^2$ is said to be outer if it is a cyclic
vector for $S$.
\end{defn}

Next we shall record a number of preliminary facts regarding inner and outer functions.
Every function other than $0$ in $H^2$ can be written as the
product of an inner function and an outer function. This factorization is
unique up to constant multipliers. A Blaschke product is an inner function of the form
\[B(z)=cz^k\prod_{j=1}^\infty \frac{\overline{\lambda}_j}{|\lambda_j|}\frac{\lambda_j-z}{1-\overline{\lambda}_jz}\]
with $k$ a non-negative integer, $|c|=1$, $\{\lambda_j\}_{j=1}^{\infty}$ a sequence of non-zero complex numbers of modulus less than 1 such that $\sum_{j=1}^{\infty}(1-|\lambda_j|)<\infty$.
A singular inner function is an inner function of the form
\[S(z)=\exp(-\int_{\mathbb{T}}\frac{w+z}{w-z}d\mu(w)),\]
where $\mu$ is a finite positive Borel measure on $\mathbb{T}$ which is singular with respect to Lebesgue measure.

It is well known that every inner function
is the product of a singular inner function and a Blaschke product.
By using the notion of inner functions, Beurling gave a complete description of the invariant subspace lattice of $S$ \cite{Beu48}.

\begin{thm}[Beurling]\label{Beurling}
A non-zero subspace $\M$ of $H^2$ is invariant under the unilateral shift $S$ if and only if $\M=\phi H^2$ for some inner function $\phi$.
\end{thm}

\begin{defn}
Let $\beta\doteq \{\beta_n\}_{n=0}^{\infty}$
 be a sequence of positive numbers with $\underset{n\geq 0}{\sup}\frac{\beta_{n+1}}{\beta_n}<\infty$. Let $S_\beta\in B(H^2)$ denote the injective unilateral weighted shift operator $S_\beta e_n=w_ne_{n+1}$, where $w_n=\frac{\beta_{n+1}}{\beta_n}$, $n\geq 0$. In particular, if $\beta_n=1$, $\forall n\in \{0\}\cup \bN$, then $S_\beta=S$.
\end{defn}

\begin{defn}
If $A,B \in B(\H)$ (and more generally, if $A,B \in \A$, a unital Banach algebra), we define the \textbf{Rosenblum operator}
\[
\begin{array}{rccc}
	\tau_{A, B}: & B(\H) & \to &B(\H), \\
		& X & \mapsto & A X - X B.
\end{array} \]
\end{defn}

\begin{prop}\label{bimodule}
Let $A,B\in B(\H)$. Then
\[\{TKR: T\in \{A\}', K\in \ker\tau_{A,B}, R\in\{B\}'\}\subset \ker \tau_{A,B},\]
where $\ker \tau_{A,B}\doteq \{X\in B(\H): \tau_{A,B}(X)=AX-XB=0 \}$.
\end{prop}

\begin{proof}
For each $T\in \{A\}'$, $R\in \{B\}'$, $K\in \ker\tau_{A,B}$, we have
\[ATKR-TKRB=T(AK-KB)R=0.\]
 Hence,
\[\{TKR: T\in \{A\}', K\in \ker\tau_{A,B}, R\in\{B\}'\}\subset \ker \tau_{A,B}.\]
\end{proof}

\begin{prop}\label{Rosenblum 1}
Let $A\in B(H^2)$, and $\beta\doteq \{\beta_n\}$
 be a sequence of positive numbers with $\beta_0=1$ and  $\underset{n\geq 0}{\sup}\frac{\beta_{n+1}}{\beta_n}<\infty$. Let $S_\beta$ denote the injective unilateral shift operator $S_\beta e_n=w_ne_{n+1}$, where $w_n=\frac{\beta_{n+1}}{\beta_n}$, $n\geq 0$.  Suppose that
 $\{\frac{\|A^n\|}{\beta_n}\}_{n=0}^{\infty}\in l^2$.
Then the following hold.
\begin{enumerate}
\item For $f\in H^2$, $K_f\doteq\sum_{n=0}^\infty \frac{A^n f}{\beta_n}\otimes e_n\in K(H^2)\cap \ker \tau_{A,S_\beta}$, and with respect to $\{e_n\}_{n=0}^\infty$,
$K_f$ has the form
\[\begin{bmatrix}
    \langle f,e_0\rangle&\langle \frac{Af}{\beta_1},e_0\rangle&\cdots&\langle \frac{A^kf}{\beta_k},e_0\rangle&\cdots\\
    \langle f,e_1\rangle& \langle\frac{Af}{\beta_1},e_1\rangle&\cdots&\langle \frac{A^kf}{\beta_k},e_1\rangle&\cdots\\
\vdots&\vdots&\ddots&\vdots&\ddots\\
\langle f,e_k\rangle& \langle \frac{Af}{\beta_1},e_k\rangle&\cdots&\langle \frac{A^kf}{\beta_k},e_k\rangle&\cdots\\
\vdots&\vdots&\ddots&\vdots&\ddots\\
\end{bmatrix}.\]

\item There is a one-one linear correspondence between $f\in H^2$ and $K\in \ker\tau_{A,S_\beta}$ such that $K=K_f$.

\item $\ker\tau_{A,S_\beta}$ is an infinite dimensional subspace of $K(H^2)$.
\end{enumerate}
\end{prop}

\begin{proof}
(1) Given $f\in H^2$. For $x\in H^2$, $m\in \bN$,
\[\|(\sum_{n=N+1}^{N+m}e_n\otimes \frac{A^nf}{\beta_n})x\|^2=\sum_{n=N+1}^{N+m}|\langle x, \frac{A^nf}{\beta_n}\rangle|^2
\leq \sum_{n=N+1}^{N+m}(\frac{\|A^n\|}{|\beta_n|})^2\|f\|^2\|x\|^2.\]
Since
\[\{\frac{\|A^n\|}{\beta_n}\}_{n=0}^{\infty}\in l^2,\] it follows that
\[\{\sum_{n=0}^{k}e_n\otimes \frac{A^nf}{\beta_n}\}_{k=0}^\infty\] is a Cauchy
sequence in $K(H^2)$, and hence so is
\[\{\sum_{n=0}^{k}\frac{A^nf}{\beta_n}\otimes e_n\}_{k=0}^\infty.\]
Thus,
\[K_f\doteq \sum_{n=0}^\infty \frac{A^nf}{\beta_n}\otimes e_n\in K(H^2).\]
 Note that
\[AK_fe_n=A(\frac{A^n}{\beta_n}f)=\frac{A^{n+1}}{\beta_n}f,\]
while
\[K_fS_\beta e_{n}=K_f\frac{\beta_{n+1}}{\beta_n}e_{n+1}=\frac{A^{n+1}}{\beta_n}f,\]
for $n\geq 0$.
Thus, $K_f\in \ker\tau_{A,S_\beta}$.
Then it is routine to check that $K_f$ has the desired matrix representation.

(2) It is straightforward to check that the corresponding between $f\in H^2$ and $K_f\in \ker\tau_{A,S_\beta}$ is linear.
If $K_f=K_g$ for $f,g\in H^2$,
then by comparing the zero-th column of the matrix representation of $K_f$ and $K_g$, we have
\[\langle f,e_n\rangle=\langle g,e_n\rangle,~~\forall n\geq 0.\]
 Hence, $f=g$. In other words, the corresponding between $f\in H^2$ and $K_f\in \ker\tau_{A,S}$ is one to one.

Conversely, if $K\in \ker\tau_{A,S_\beta}$, then for each $n\geq 0$,
\[
Ke_{n+1}=\frac{\beta_{n}}{\beta_{n+1}}KS_\beta e_n=\frac{\beta_{n}}{\beta_{n+1}}AKe_n.
\]
Let $f=Ke_0$. Then by induction, we have
\[Ke_n=\frac{\beta_0}{\beta_n}A^nKe_0=\frac{A^nf}{\beta_n}=K_fe_n,~~n\geq 0.\]
Therefore, $K=K_f$.

(3) The proof follows immediately from (2) and the fact that $\ker \tau_{A,S_\beta}$ is closed.
\end{proof}

From now on, we call such operator $K_x$ as a \textbf{shift representation operator}. With the help of $K_x$, we could find some hidden analytic structure and geometric structure of $A$.
The following result will prove very useful.

\begin{prop}\label{two conditions make A intransitive}
Suppose that $A\in B(H^2)$ with $r(A)<1$, and $0\neq K\in \ker\tau_{A,S}$. Then the following statements hold.
\begin{enumerate}
\item If $0\in \sigma_p(K)$, then $A$ has eigenvalues.
\item If $\overline{\ran K}\neq H^2$, then $A$ is intransitive.
\end{enumerate}
\end{prop}

\begin{proof}
By Proposition \ref{Rosenblum 1}, there exists a unique nonzero vector $f\in H^2$, such that
\[K=K_f\doteq \sum_{n=0}^\infty A^nf\otimes e_n.\]

(1) Set $\M\doteq \ker K_f$. Since $0\neq K_f$ and $0\in \sigma_p(K_f)$, $\M$ is a nontrivial invariant subspace of $S$
as $AK_f=K_fS$. By Theorem \ref{Beurling}, there exists an inner function $\phi\in H^\infty$, such that $\M=\phi H^2$.
Write
\[\phi=\sum_{n=0}^\infty \alpha_nz^n.\]
Note that $r(A)<1$. Then
\[T\doteq \phi(A)=\sum_{n=0}^\infty \alpha_n A^n\in B(H^2).\]
 Since $K_f\phi=0$, we have
\[\sum_{n=0}^\infty \alpha_n A^nf=0,\]
 and hence $Tf=0$.

We claim that
$\phi$ is not a singular inner function.
Otherwise,
 $\phi(z)\neq 0$, $\forall z\in \bD$,  as a singular inner function has no zero points in $\mathbb{D}$. By the spectral mapping theorem,
\[0\notin \phi(\sigma(A))=\sigma(\phi(A))=\sigma(T).\]
This contradicts with $f\in \ker T$. Thus, $\phi$ is not a singular inner function.
So $\phi$ has a zero point $z_0\in \mathbb{D}$. Then there exists inner function $\psi\in H^\infty$,
such that
\[\phi=\frac{z_0-z}{1-\overline{z}_0z}\psi.\]
Set $\N=\psi H^2$.
Note that $\dim(\N\ominus \M)=1$ and $K(\M)=\{0\}$. It follows that
\[1\leq \dim K(\N)=\dim K(\M)+\dim K(\N\ominus \M)=\dim K(\N\ominus \M)\leq 1.\]
Observe that
\[A(K(\N))=K(S(\N))\subset K(\N).\]
Hence, $A$ has eigenvalues.

(2) By $AK=KS$, we have
\[A\overline{\ran K}\subset \overline{\ran K}\neq  H^2.\] Thus, $\overline{\ran K}$ is a nontrivial invariant subspace of $A$. That is, $A$ is intransitive.
\end{proof}

\begin{thm}\label{main theorem 0}
Let $A\in B(H^2)$ with $r(A)<1$. If there exist $0\neq \{c_n\}_{n=0}^\infty \in l^2$ and $0\neq f\in H^2$ such that $\sum_{n=0}^\infty c_n A^nf=0$,
then $A$ has eigenvalues.
\end{thm}

\begin{proof}
Since $f\neq 0$, by Proposition \ref{Rosenblum 1},
\[K_f=\sum_{n=0}^\infty A^nf\otimes e_n\neq 0.\]
Since $0\neq \{c_n\}_{n=0}^\infty \in l^2$,
\[h\doteq \sum_{n=0}^\infty c_nz^n\]
is nonzero in $H^2$. Noting that
\[\sum_{n=0}^\infty c_n A^nf=0,\]
we have $K_fh=0$. By Proposition \ref{two conditions make A intransitive},
$A$ has eigenvalues.
\end{proof}

\begin{cor}
Suppose that $A\in B(H^2)$ with $r(A)<1$, and $0\neq K\in \ker\tau_{A,S}$. If there exists $0\neq \M\in \textup{Lat}S$
such that $\overline{K(\M)}\neq H^2$, then $A$ is intransitive.
\end{cor}

\begin{proof}
According to Proposition \ref{two conditions make A intransitive}, we may assume that $K$ is injective. By Theorem \ref{Beurling}, there exists an inner function $\phi\in H^\infty$, such that $\M=\phi H^2$. Set $T_\phi:H^2\to H^2, f\mapsto \phi f$.
 Then $T_\phi\in \{S\}'$ is an isometry, and $K\in \ker\tau_{A,S}$. It follows from Proposition \ref{bimodule} that $0\neq KT_\phi\in \ker\tau_{A,S}$. Since
\[\overline{\ran KT_\phi}=\overline{K(\M)}\neq H^2.\]
By Proposition \ref{two conditions make A intransitive}, $A$ is intransitive.
\end{proof}

\section{Another version of the Gelfand-Hille Theorem and A refined version of Wermer's Theorem }

In this section, we will give a generalization of Wermer's Theorem.  Furthermore, we also show the existence of nontrivial invariant subspaces for compact perturbations of some quasinilpotent operators
by establishing another version of the Gelfand-Hille Theorem.

The following Lemma is quite crucial. It shows that the growth condition can be weakened in an
asymmetric way.

\begin{lem}\label{L:crucial lemma}
Let $T\in B(\H)$ be an invertible operator. Suppose that there exists a positive integer $r$, a nonzero complex constant $\lambda$ such that
\[\|T^{-n}+\lambda T^{n}\|\sim O(n^r),~~n\to +\infty.\]
Then $\exists M>0$, such that
\[\|T^n(I-T^2)\|\leq M|n|^{r+2},~~\forall n\in \mathbb{Z}\setminus\{0\}.\]
\end{lem}

\begin{proof}
Let $\{e_n\}_{n=0}^\infty$ be an orthonormal basis of $\H$.
For $n\in \mathbb{N}\cup\{0\}$, set
\[\beta_n=(n+1)^{r+2}.\]
Then it is easy to see that
\[\{\frac{\|T^{-(n+1)}+\lambda T^{(n+1)}\|}{\beta_n}\}_{n=0}^{\infty}\in l^2.\]
For $x\in \H$, by the similar method used in the proof of Proposition \ref{Rosenblum 1}, one could easily show that
\[\tilde{K}_x=\sum_{n=0}^\infty\frac{(T^{-(n+1)}+\lambda T^{(n+1)})x}{\beta_n}\otimes e_n\in K(\H).\]

Consider the weighted shift operator $S_\beta\in B(\H)$ which satisfies that
\[S_\beta e_n=\frac{\beta_{n+1}}{\beta_n}e_{n+1},~~n\geq 0.\]
Now, we get that
\[\tau_{T^{-1},S_\beta}(\tilde{K}_x)=T^{-1}\tilde{K}_x-\tilde{K}_xS_\beta\in K(\H).\]
Noting that $e_n\overset{w}{\to} 0$, we can deduce that
\begin{align*}
0=\underset{n\to \infty}{\lim}\|(T^{-1}\tilde{K}_x-\tilde{K}_xS_\beta)e_n\|
&=\underset{n\to \infty}{\lim}\|\frac{(T^{-(n+2)}+\lambda T^n)x}{\beta_n}-\frac{(T^{-(n+2)}+\lambda T^{(n+2)})x}{\beta_n}\|\\
&=\underset{n\to \infty}{\lim}|\lambda|\|\frac{T^n(I-T^2)x}{\beta_n}\|.
\end{align*}
Hence, the Uniform Boundedness Principle implies the existence of a positive number $M_1$
such that
\[\|\frac{T^n(I-T^2)}{\beta_n}\|\leq M_1,~~\forall n\geq 1.\]
Hence,
\[\|T^n(I-T^2)\|\leq M_1\cdot \beta_n=M_1\cdot (n+1)^{r+2}\leq 2^{r+2}M_1\cdot n^{r+2},~~\forall n\geq 1.\]

By the symmetry of $T^{-1}$ and $T$, there exists a positive number $M_2$,
such that
 \[\|(T^{-1})^n(I-(T^{-1})^2)\|\leq M_2\cdot n^{r+2},~~\forall n\geq 1.\]
Thus,
\[\|T^{-n}(I-T^2)\|\leq M_2\cdot\|T^2\|\cdot n^{r+2},~~\forall n\geq 1.\]
Set
\[M=\max\{2^{r+2}M_1, M_2\cdot\|T^2\|\}.\]
Then the proof is completed.
\end{proof}

The following Corollary is a local version of Lemma \ref{L:crucial lemma} which could be proved similarly.

\begin{cor}\label{crucial corollary}
Let $T\in B(\H)$ be an invertible operator. Suppose that there exist $x\in \H$ and a positive integer $r$, and a nonzero complex constant $\lambda$ such that
\[\|(T^{-n}+\lambda T^{n})x\|\sim O(n^r),~~~~n\to  +\infty.\]
Then $\exists M>0$, such that
\[\|T^n(I-T^2)x\|\leq M|n|^{r+2},~~\forall n\in \mathbb{Z}\setminus\{0\}.\]
\end{cor}

Recall that a subspace $\mathcal{M}$ is {\it hyperinvariant} for $A$ if $\mathcal{M}$ is invariant under every $B\in \{A\}'$.  The following theorem is a version of \cite[Theorem 1.2]{Atz84} due to Atzmon.

\begin{thm}\label{Atzmon}
Let $A$ be an invertible operator in $B(\H)$ and let $x_0,y_0\in \H$ be non zero vectors.
Suppose that
\[\|A^nx_0\|+\|A^{*n}y_0\|\sim O(|n|^r),~~n\to \pm \infty,\]
for some integer $r>0$.
Then $A$ is a multiple of the identity operator or $A$ has a non trivial hyperinvariant subspace.
\end{thm}

We next state a refined version of Wermer's Theorem.

\begin{thm}\label{T:Another version of wermer}
Let $A$ be an invertible operator in $B(\H)$ and let $x,y\in \H$ be non zero vectors.
Suppose that
\[\|A^{-n}x+\lambda A^{n}x\|\sim O(n^r),~~\forall n\in \mathbb{N};\]
and
\[\|(A^*)^{-n}y+\mu (A^*)^{n}y\|\sim O(n^r),~~\forall n\in \mathbb{N};\]
for some integer $r>0$, and nonzero complex constants $\lambda,\mu$.
Then $A$ is a multiple of the identity operator or $A$ has a non trivial hyperinvariant subspace.
\end{thm}

\begin{proof}
We may suppose that $A$ is not a multiple of the identity operator.
If $(I-A^2)x=0$ or $(I-(A^*)^2)y=0$, it is well known that $A$
has a non trivial hyperinvariant subspace. Hence, we could further assume
that
\[x_0\doteq (I-A^2)x\neq 0,~~y_0\doteq (I-(A^*)^2)y\neq 0.\]
By Corollary \ref{crucial corollary},
$\exists M>0$, such that
\[\|A^nx_0\|=\|A^n(I-A^2)x\|\leq M|n|^{r+2},~~\forall n\in \mathbb{Z}\setminus\{0\},\]
and
\[\|A^{*n}y_0\|=\|A^{*n}(I-A^{*2})y\|\leq M|n|^{r+2},~~~~\forall n\in \mathbb{Z}\setminus\{0\}.\]
By Theorem \ref{Atzmon}, $A$ has a non trivial hyperinvariant subspace.
\end{proof}

We are now in place to obtain Theorem \ref{A}.

\begin{proof}[\bf Proof of Theorem A]
Since
\[\|A^{-n}+\frac{1}{\lambda} A^{n}\|=\frac{1}{\lambda}\|A^n+\lambda A^{-n}\|\sim O(n^r),~~\forall n\in \bN,\]
clearly
\[\|(A^*)^{-(n)}+\frac{\lambda}{|\lambda|^2} (A^*)^{n}\|\sim O(n^r),~~\forall n\in \bN.\]
By Theorem \ref{T:Another version of wermer},  $A$ has a non trivial hyperinvariant subspace. Since $T\in \{A\}'$, $T$ is intransitive.
\end{proof}


Our second use of Lemma \ref{L:crucial lemma} is to obtain the following version of Gelfand-Hille Theorem.

\begin{thm}\label{NEW Gelfand-Hille}
Let $A\in B(\H)$ and $\sigma(A)=\{1\}$. Suppose that there exist a positive integer $r$, a nonzero complex constant $\lambda$ such that
\[\|A^{-n}+\lambda A^{n}\|\sim O(n^r),~~n\to \infty.\]
Then $(A-I)^{r+4}=0$.
\end{thm}

\begin{proof}
By Lemma \ref{L:crucial lemma}, There exists $M>0$, such that
\[\|A^n(I-A^2)\|\leq M|n|^{r+2},~~\forall n\in \mathbb{Z}\setminus\{0\}.\]
Equivalently,
\[\|A^n(A-I)(A+I)\|\leq M|n|^{r+2},~~\forall n\in \mathbb{Z}\setminus\{0\}.\]
Since $\sigma(A+I)=\{2\}$, it follows that
\[\|A^{n+1}-A^n\|\leq M\|(A+I)^{-1}\||n|^{r+2},\forall n\in \mathbb{Z}\setminus\{0\}.\]
By Theorem 4 of \cite{Zem94},
\[(A-I)^{r+4}=0.\]
\end{proof}

Finally, we are ready to give the proof of Theorem \ref{B}.

\begin{proof}[\bf Proof of Theorem B]
Set $A=T+K+I$, then $\sigma(A)=\{1\}$.
Now
\[\|A^{-n}+\frac{1}{\lambda} A^{n}\|=\frac{1}{\lambda}\|A^n+\lambda A^{-n}\|\sim O(n^r),~~\forall n\in \bN,\]
By Theorem \ref{NEW Gelfand-Hille},
\[(A-I)^{r+4}=0.\]
In other words,
\[(T+K)^{r+4}=0.\]
Hence, $T$ is a polynomially compact operator. Then according to \cite[Corollary 5.6]{RR2}, $T$ is intransitive.
\end{proof}

\section{Another equivalent of the invariant subspace problem}

In this section we shall be applying the properties of the shift representation operators, Proposition \ref{Rosenblum 1} and Proposition \ref{two conditions make A intransitive} to
prove Theorem \ref{C}, which is another equivalent of the Invariant Subspace Problem via the injectivity of certain Hankel operators.

\begin{defn}
Let $P$ denote the orthogonal projection from $L^2(\mathbb{T})$ onto $H^2$. For $\varphi=\sum_{n=-\infty}^{\infty}c_nz^n\in L^\infty(\mathbb{T})$,
the Toeplitz operator $T_\varphi$ on $H^2$ with symbol $\varphi$ is defined as
\[T_\varphi g=P(\varphi g),~~\forall g\in H^2,\]
the Hankel operator $H_\varphi$ from $H^2$ to $H^2$ with symbol $\varphi$ is defined as
\[H_\varphi g=P(J(\varphi g)),~~\forall g\in H^2,\]
where $J\in B(L^2(\mathbb{T}))$ is defined by
\[Jf(z)=f(\overline{z}),~~\forall f\in L^2(\mathbb{T}).\]
It is easy to check that $J$ is a symmetry, i.e. $J=J^*=J^{-1}$.
\end{defn}

The following theorem is derived from Peller's monograph \cite{Pel}.

\begin{thm}\cite[Theorem 2.3]{Pel}\label{injectivity of Hankel operator}
Let $\psi\in L^\infty(\mathbb{T})$. Then the following are equivalent:
\begin{enumerate}
\item[(i)] $H_\psi$ has nontrivial kernel;
\item[(ii)] $\ran H_\psi$ is not dense in $H^2$;
\item[(iii)] $\psi=\overline{\vartheta}\phi$ for some inner function $\vartheta$
 and some function $\phi\in H^\infty$.
 \end{enumerate}
\end{thm}

\begin{prop}\label{alpha}
Let $S$ be the unilateral shift operator on $H^2$ and $A\in B(H^2)$ with $r(A)<1$, $f$ be a unit vector of $H^2$. Suppose that there exists a nonzero operator $G\in \{S^*\}'$,
such that
$\alpha\in \ker G$,
where
\[\alpha=\sum_{n=0}^\infty
\overline{\langle A^nf,e_0\rangle}z^n.\]
Then $A$ is intransitive.
\end{prop}

\begin{proof}
If $\alpha=0$, then either $A^Nf=0$ for some $N\in \bN$ or $e_0\perp A^nf$, $A^nf\neq 0$, $\forall n\geq 0$.
In each case, $A$ is intransitive. Hence, we may assume that $\alpha\neq 0$.

By Proposition \ref{Rosenblum 1},
\[K_f\doteq \sum_{n=0}^\infty A^nf\otimes e_n\in \ker\tau_{A,S}.\] Note that
\[K_f^*=\sum_{n=0}^\infty e_n\otimes A^nf,\] and hence with respect to $\{e_n\}_{n=0}^\infty$,
\[K_f^*=\begin{matrix}
\begin{bmatrix}
    \overline{\langle f,e_0\rangle}&\overline{\langle f,e_1\rangle}&\cdots&\overline{\langle f,e_k\rangle}&\cdots\\
    \overline{\langle Af,e_0\rangle}& \overline{\langle Af,e_1\rangle}&\cdots&\overline{\langle Af,e_k\rangle}&\cdots\\
\vdots&\vdots&\ddots&\vdots&\ddots\\
\overline{\langle A^kf,e_0\rangle}& \overline{\langle A^kf,e_1\rangle}&\cdots&\overline{\langle A^kf,e_k\rangle}&\cdots\\
\vdots&\vdots&\ddots&\vdots&\ddots\\
\end{bmatrix}&\begin{matrix}
  e_0\\
  e_1\\
  \vdots\\
e_k\\
  \vdots\\
\end{matrix}
\end{matrix}.\]
Since $G\alpha=0$, with respect to $\{e_n\}_{n=0}^\infty$,
\[GK_f^*=\begin{matrix}
\begin{bmatrix}
    0&\ast&\cdots&\ast&\cdots\\
    0&\ast&\cdots&\ast&\cdots\\
\vdots&\vdots&\ddots&\vdots&\ddots\\
0&\ast&\cdots&\ast&\cdots\\
\vdots&\vdots&\ddots&\vdots&\ddots\\
\end{bmatrix}&\begin{matrix}
  e_0\\
  e_1\\
  \vdots\\
e_k\\
  \vdots\\
\end{matrix}
\end{matrix};\]
therefore,
\[K_fG^*=\begin{matrix}
\begin{bmatrix}
    0&0&\cdots&0&\cdots\\
    \ast&\ast&\cdots&\ast&\cdots\\
\vdots&\vdots&\ddots&\vdots&\ddots\\
\ast&\ast&\cdots&\ast&\cdots\\
\vdots&\vdots&\ddots&\vdots&\ddots\\
\end{bmatrix}&\begin{matrix}
  e_0\\
  e_1\\
  \vdots\\
e_k\\
  \vdots\\
\end{matrix}
\end{matrix}.\]
Since $G\in \{S^*\}'$, we have $G^*\in \{S\}'$.
By Proposition \ref{bimodule},
$K_fG^*\in \ker \tau_{A,S}$.

Since $\alpha\neq 0$, $K_f\neq 0$. If $\ker K_f\neq \{0\}$, by Proposition \ref{two conditions make A intransitive} (1), $A$ is intransitive.
If $\ker K_f= \{0\}$, since $G^*\neq 0$, $K_fG^*\neq 0$. By observing the matrix representation of $K_fG^*$,
one could deduce that $\overline{\ran K_fG^*}\neq H^2$. Now by Proposition \ref{two conditions make A intransitive} (2), $A$ is intransitive.
\end{proof}

\begin{prop}
Let $f,g\in H^2$ and $A\in B(H^2)$ with $r(A)<1$.
Then $\psi=\sum_{n=0}^\infty \langle A^nf,g\rangle z^{-n}\in L^\infty(\mathbb{T})$ and $H_\psi$ is a Hilbert-Schmidt operator.
\end{prop}

\begin{proof}
Since $r(A)<1$, there exists a positive number $\delta>0$ such that $r(A)<1-\delta$. Then by the spectral radius formula, we could find a positive integer
$N$ such that $\|A^n\|\leq (1-\delta)^n$ for each $n\geq N$. Then
\[\psi=\sum_{n=0}^\infty \langle A^nf,g\rangle z^{-n}\in L^\infty(\mathbb{T}),\]
accordingly $H_\psi$ is a Hankel operator.

Note that
\[\sum_{m,n=0}^{\infty}|\langle H_\psi e_n,e_m \rangle|^2=\sum_{n=0}^\infty (n+1)|\langle A^nf,g\rangle|^2\leq \|f\|^2\|g\|^2 (\sum_{n=0}^\infty (n+1)\|A^n\|^2).\]
Since $r(A)<1$,  we have
\[\underset{n\to \infty}{\lim}[(n+1)\|A^n\|^2]^{\frac{1}{n}}=r(A)^2<1.\]
Hence, $H_\psi$ is a Hilbert-Schmidt operator.
\end{proof}

Now, we are at the position to obtain Theorem \ref{C}.

\begin{proof}[\bf Proof of Theorem C]
The necessary part is easy. In fact, since $A$ is intransitive, we could pick a nontrivial invariant subspace $\M$ of $A$.
Let $f\in \M$, $g\in \M^\perp$ be unit vectors. Now $\langle A^nf,g\rangle=0$, $\forall n\geq 0$.
Hence, $H_\psi=0$. In particular, $H_\psi$ is not injective.

Now we focus on the sufficient part. With no loss of generality, we may assume that $A$ is injective and $f,g$ are unit vectors.
It will be convenient to consider the following two cases separately.
\begin{enumerate}

\item[\textsc{Case}] \textsc{1.}
If $H_\psi=0$, then $g\perp \bigvee_{n=0}^\infty\{A^nf\}\doteq \M$. Since $A$ is injective, $\M$ is a nontrivial invariant subspace of
$A$. Therefore, $A$ is intransitive.

\item[\textsc{Case}] \textsc{2.}
If $\{0\}\subsetneq\ker H_\psi\subsetneq H^2$.
Let $U\in B(H^2)$ be a unitary operator such that $Ue_0=g$.
Now
\[\psi=\sum_{n=0}^{\infty}\langle A^n f,g\rangle z^{-n}=\sum_{n=0}^{\infty}\langle U^*A^n U(U^*)f,e_0\rangle z^{-n}=
\sum_{n=0}^{\infty}\langle(U^*AU)^n (U^*f),e_0\rangle z^{-n}.\]
Thus, by replacing $A$ and $f,g$ by $U^*AU$ and $U^*f, U^*g$, we
may assume without loss of generality that $g=e_0$.

We claim that there exists an inner function $\vartheta$,
such that $H_\psi\vartheta=0$. As $H_\psi$ is a Hankel operator, by using the fact that $S^*H_\psi=H_\psi S$, one could obtain that $\{0\}\subsetneq\ker H_\psi\in \textup{Lat}S$. By Theorem \ref{Beurling}, there exists an inner function $\vartheta=\sum_{n=0}^\infty b_ne_n$,
such that $\ker H_\psi=\vartheta H^2$. In particular, $H_\psi\vartheta=0$. That is,
\[
\begin{matrix}\begin{bmatrix}
    \langle f,e_0\rangle &\langle Af,e_0\rangle&\langle A^2f,e_0\rangle&\langle A^3f,e_0\rangle&\cdots\\
     \langle Af,e_0\rangle &\langle A^2f,e_0\rangle&\langle A^3f,e_0\rangle&&\ddots\\
\langle A^2f,e_0\rangle &\langle A^3f,e_0\rangle&&\ddots&\\
\langle A^3f,e_0\rangle &&\ddots&&\\
\vdots&\ddots&&&\\
\end{bmatrix}&\begin{bmatrix}
  b_0\\
  b_1\\
  b_2\\
 \vdots\\
  \vdots\\
\end{bmatrix}\end{matrix}=\begin{bmatrix}
 0\\
  0\\
  0\\
\vdots\\
  \vdots\\
\end{bmatrix}.\]
Therefore,
\[\sum_{n=0}^\infty\langle A^{l+n}f,e_0\rangle b_n=0,~~\forall l\geq 0.\]
Thus,
\[\sum_{n=0}^\infty \overline{b_n}\cdot\overline{\langle A^{l+n}f,e_0\rangle}=0,~~\forall l\geq 0.\]
Consequently,
\[
\begin{bmatrix}
    \overline{b_0}&\overline{b_1}&\overline{b_2}&\overline{b_3}&\cdots\\
  &\overline{b_0}&\overline{b_1}&\overline{b_2}&\ddots\\
 &&\overline{b_0}&\overline{b_1}&\ddots\\
 &&&\overline{b_0}&\ddots\\
&&&&\ddots\\
\end{bmatrix}\begin{bmatrix}
\overline{\langle f,e_0\rangle}\\
\overline{\langle Af,e_0\rangle}\\
\overline{\langle A^2f,e_0\rangle}\\
\overline{\langle A^3f,e_0\rangle}\\
\vdots\\
\end{bmatrix}=0.\]
Therefore, $T_\vartheta^*\alpha=0$, where
\[\alpha=\sum_{n=0}^\infty \overline{\langle A^nf,e_0\rangle}e_n.\] Note that $T_\vartheta^*\in \{S^*\}'$ is a co-isometry, and $\alpha\neq 0$ as $H_\psi\neq 0$.
By Proposition \ref{alpha},
$A$ is intransitive.

\end{enumerate}

\end{proof}

\begin{eg}
Let $A\in B(H^2)$ be an algebraic operator. Since $A$ is intransitive if and only if $\frac{A}{\|A\|+1}$ is intransitive. We may assume $r(A)<1$ with no loss of generality. Then for any vectors $f,g\in H^2$, $H_\psi$ is not injective, where $\psi=\sum_{n=0}^{\infty}\langle A^n f,g\rangle z^{-n}$.

In fact, since $A$ is an algebraic operator, there exists $N\in \bN$, and $(\alpha_0,\alpha_1,\cdots,\alpha_N)\in \bC^N\setminus\{0\}$, such that $\sum_{n=0}^N \alpha_n A^n=0$. Note that
\[\begin{bmatrix}
    \langle f,g\rangle &\langle Af,g\rangle&\langle A^2f,g\rangle&\langle A^3f,g\rangle&\cdots\\
     \langle Af,g\rangle &\langle A^2f,g\rangle&\langle A^3f,g\rangle&&\ddots\\
\langle A^2f,g\rangle &\langle A^3f,g\rangle&&\ddots&\\
\langle A^3f,g\rangle &&\ddots&&\\
\vdots&\ddots&&&\\
\end{bmatrix}\begin{bmatrix}a_0\\
a_1\\
\vdots\\
a_N\\
0\\
\vdots\\
\end{bmatrix}=\begin{bmatrix}
\langle \sum_{n=0}^N \alpha_nA^nf,g\rangle\\
\langle A(\sum_{n=0}^N \alpha_nA^nf),g\rangle\\
\vdots\\
\langle A^k(\sum_{n=0}^N \alpha_nA^nf),g\rangle\\
\vdots\\
\vdots\\
\end{bmatrix}=0.\]
Hence $H_{\psi}\alpha=0$,
where $\alpha=\sum_{n=0}^N \alpha_n e_n$.
Therefore, by Theorem \ref{C}, $A$ is intransitive.

\end{eg}

We now give two applications of Theorem \ref{C}.

\begin{cor}
Let $A\in B(H^2)$ with $r(A)<1$.
Then $A$ is intransitive if one of the following conditions holds.
\begin{enumerate}
\item there are two nonzero vectors $f,g\in H^2$ and a constant $\alpha\in \bC$, $k\in \bN$
such that $\langle A^nf,g\rangle=\alpha \langle A^{n+1}f,g\rangle$, $\forall n\geq k$.
\item there are two nonzero vectors $f,g\in H^2$, such that
the $|\psi|$ is a constant $c$, where $\psi=\sum_{n=0}^{\infty}\langle A^n f,g\rangle z^{-n}$.
\end{enumerate}
\end{cor}

\begin{proof}

(1). By assumption, $H_\psi$ is a finite rank Hankel operator, and hence $\ker H_\psi\neq \{0\}$, where $\psi=\sum_{n=0}^\infty \langle A^nf,g\rangle z^{-n}$.
Therefore, by Theorem \ref{C}, $A$ is intransitive.

(2). If $c=0$, then $H_\psi=0$, hence $A$ is intransitive by Theorem \ref{C}.
Thus, we assume that $c\neq 0$. Then
\[\psi(z)=\overline{\theta(z)}\cdot c,\]
where $\theta(z)=\frac{\overline{\psi(z)}}{c}$.
Note that $\theta(z)$ is an inner function.
By Theorem \ref{injectivity of Hankel operator}, $\ker H_\psi\neq \{0\}$. So $A$ is intransitive by Theorem \ref{C}.
\end{proof}

\begin{cor}
Let $A\in B(H^2)$ with $r(A)<1$. Then $A$ is intransitive if there are two nonzero vectors $f,g\in H^2$, a nonnegative integer $k$, such that
the Hankel operator $H_\gamma$ is not injective, where $\gamma=\sum_{n=0}^{\infty}\langle A^{n+k} f,g\rangle z^{-n}$.
\end{cor}

\begin{proof}
If $A^kf=0$, then $H_\psi$ is a finite rank operator, where $\psi=\sum_{n=0}^{\infty}\langle A^{n} f,g\rangle z^{-n}$.
Hence, by Theorem \ref{C}, $A$ is intransitive.

If $A^kf\neq 0$, then set $f_1=A^kf$. Now $H_{\psi_1}=H_\gamma$ is not injective, where
\[\psi_1=\sum_{n=0}^{\infty}\langle A^{n} f_1,g\rangle z^{-n}.\] By Theorem \ref{C} again, $A$ is intransitive.

\end{proof}

\section{Cyclic vectors}

We first need an elegant result about the cyclic vectors of $S^*$, due to Douglas, Shapiro and Shields \cite{DSS}.

\begin{thm}\cite[Theorem 2.2.4]{DSS}\label{DSS}
If $f$ is holomorphic in $|z|<R$ for some
$R>1$, then $f$ is either a cyclic vector of $S^*$ or a rational function $($and hence non-cyclic$)$.
\end{thm}

\begin{lem}\label{nu}
Let $A\in B(H^2)$ with $r(A)<1$, $f,g\in H^2$ be nonzero vectors.
Then $H_\psi$ is not injective if and only if $\nu=\sum_{n=0}^{\infty}\langle A^n f,g\rangle z^{n}$ is a rational function,
where $\psi=\sum_{n=0}^{\infty}\langle A^n f,g\rangle z^{-n}$.
\end{lem}

\begin{proof}
Since $r(A)<1$,
\[\underset{n\to \infty}{\overline{\lim}}|\langle A^n f,g\rangle|^{\frac{1}{n}}\leq r(A).\]
 Hence,
\[\nu=\sum_{n=0}^{\infty}\langle A^n f,g\rangle z^{n}\] is holomorphic in $|z|<\frac{1}{r(A)}$.

Observe that with respect to $\{e_n\}_{n=0}^{\infty}$,
\[H_\psi=\begin{bmatrix}
    \langle f,g\rangle &\langle Af,g\rangle&\langle A^2f,g\rangle&\langle A^3f,g\rangle&\cdots\\
     \langle Af,g\rangle &\langle A^2f,g\rangle&\langle A^3f,g\rangle&&\ddots\\
\langle A^2f,g\rangle &\langle A^3f,g\rangle&&\ddots&\\
\langle A^3f,g\rangle &&\ddots&&\\
\vdots&\ddots&&&\\
\end{bmatrix}=[\nu, S^*\nu,(S^*)^2\nu, (S^*)^3\nu,\cdots].\]
Hence,
\[\overline{\ran H_\psi}=\bigvee_{n=0}^\infty \{(S^*)^n\nu\}.\]
That is, $\overline{\ran H_\psi}\neq H^2$ if and only if $\nu$ is not a cyclic vector of $S^*$.

According to Theorem \ref{DSS}, $\nu$ is not a cyclic vector of $S^*$ if and only if $\nu$ is a rational function.
By Theorem \ref{injectivity of Hankel operator}, $H_\psi$ is not injective if and only if $\overline{\ran H_\psi}\neq H^2$.

Therefore, $H_\psi$ is not injective if and only if $\nu$ is a rational function.

\end{proof}

\begin{eg}
Let $\{e_n\}_{n=0}^\infty$ be the standard orthonormal basis of $H^2$, and $A$ be the weighted unilateral shift operator such that $Ae_n=\frac{1}{n+1}e_{n+1}$.
Then it is easy to see that $\sigma(A)=\{0\}$ and $A$ is intransitive.
Let $f=e_0$ and $g=\sum_{n=0}^{\infty}\frac{1}{n+1}e_n$.
Since
\[\nu=\sum_{n=0}^\infty \langle A^nf, g\rangle z^n=\sum_{n=0}^\infty \frac{1}{(n+2)!}z^n\]
 is an entire function. By Lemma \ref{nu}, $H_{\psi}$ is injective, where $\psi=\sum_{n=0}^{\infty}\langle A^n f,g\rangle z^{-n}$.

However, let $f_1=e_1$, $g_1=e_0$, then
\[\psi_1\doteq \sum_{n=0}^{\infty}\langle A^n f_1,g_1\rangle z^{-n}=0,\]
and hence $H_{\psi_1}=0$.
Thus, by Theorem \ref{C}, $A$ is intransitive.

We remark that when one intends to use Theorem \ref{C} to show that $A$ being intransitive, one need choose
$f,g$ carefully!

\end{eg}

The following theorem is a variant of Theorem \ref{C}.

\begin{thm}\label{T:variant of C}
Let $A\in B(H^2)$ with $r(A)<1$. Then $A$ is intransitive if and only if there are nonzero vectors $f,g\in H^2$, such that $\nu=\sum_{n=0}^{\infty}\langle A^n f,g\rangle z^{n}$ is a rational function.
\end{thm}

\begin{proof}
This follows by combining Theorem \ref{C} with Lemma \ref{nu}.

\end{proof}

Now we are ready to give the proof of Theorem \ref{D}.
\begin{proof}[\textup{\textbf{Proof of Theorem \ref{D}}}]
Suppose toward contradiction that there exists  $0\neq\alpha=\{\alpha_k\}_{k=0}^{\infty}$, $\sum_{k=0}^{\infty}\alpha_kT^kx$ is a
non-cyclic vector of $T$. Let $g(z)=\sum_{k=0}^\infty \alpha_k z^k$. Then $g(z)$ is a nonzero analytic function at zero. By Proposition \ref{two conditions make A intransitive},
$K_x$ is injective as $\sigma_p(T)=\emptyset$.
Hence,
\[g(T)x=K_xg\neq 0.\]
Set
\[\M=\bigvee_{k=0}^\infty\{S^kg\},~~\N=\bigvee_{k=0}^{\infty}\{T^k(g(T)x)\}.\]
Since $g(T)x$ is a non-cyclic vector of $T$, we could choose a unit vector $y\in H^2$ such that $y\perp \N$.
Furthermore, let $f=K_x^*y$. Notice that
\[TK_x=K_xS,\]
 and hence
\[T^k(g(T)x)=T^k(K_xg)=K_xS^kg,~~\forall k\geq 0.\]
Then for $k\geq 0$,
\[\langle (S^*)^kf, g\rangle=\langle f, S^kg\rangle=\langle K_x^*y, S^kg\rangle=\langle y, K_xS^kg\rangle=\langle y,  T^k(g(T)x)\rangle=0,\]
which implies that  $f$ is a non-cyclic vector for $S^*$.

Notice that
\[f=K_x^*y=(\sum_{k=0}^\infty T^kx\otimes e_k)^*y=\sum_{k=0}^\infty \langle y,T^kx\rangle e_k=\sum_{k=0}^\infty \langle y,T^kx\rangle z^k.\]
Since $\sigma(T)=\{0\}$, as done in the proof of Lemma \ref{nu}, it is straightforward to check that $f$ is an entire function. Then by Theorem \ref{DSS}, $f$ being a non-cyclic vector of $S^*$ implies that $f$ is a polynomial, which contradicts with the assumption that
$T^mx$ is a cyclic vector of $T$ for each $m\in \bN$.

\end{proof}

Next we will introduce a straightforward application of Theorem \ref{D}.
Recall that the Volterra operator on $L^2[0,1]$ is defined by
\[Vf(t)=\int_0^tf(s)\textup{d}s, ~t\in[0,1],~\forall f\in L^2[0,1].\]
 $V$ is a distinguished example of quasinilpotent operator, and it has been extensively studied. A classical reference is \cite{GK70}.
It is well known that $\sigma_p(V)=\emptyset$. By the Weierstrass approximation theorem, one could conclude that $V^m(1)$ is a cyclic vector of $V$, $\forall m\in \bN$.
Therefore, by Theorem \ref{D}, for $0\neq\alpha=\{\alpha_k\}_{k=0}^{\infty}\in l^2$,
\[\sum_{k=0}^{\infty}\alpha_k V^k(1)=\sum_{k=0}^{\infty}\alpha_k \frac{t^k}{k!}\]
is a cyclic vector of $T$.

In 1970, Halmos raised ten open problems on
operator theory. A companion problem of the third problem of Halmos is the following \cite{Hal70}.
Let $T\in B(\H)$ such that its square $T^2$
is intransitive, is $T$ intransitive?
Applying Theorem \ref{T:variant of C}, we provide a partial result about this problem.

\begin{prop}
Let $T\in B(\H)$ with $r(T)<1$. Suppose that $\M\in \textup{Lat}T^2$, and $\M\neq \{0,\H\}$.
Suppose that there exist $0\neq x\in \M $, $0\neq y\in \M^\perp$, $|\lambda_i|>1$, $1\leq i\leq d$,
such that
\[\langle T^{2k}x,T^*y\rangle =\sum_{i=1}^d \frac{1}{\lambda_i^k},\]
where $d\in \bN$. Then $T$ is intransitive.
\end{prop}

\begin{proof}
Since $\M\in \textup{Lat}T^2$, $x\in \M $, $y\in \M^\perp$, then
\[\langle T^{2k}x,y\rangle=0,~~\forall k\geq 0.\]
Then
\begin{align*}
\sum_{k=0}^\infty \langle T^{k}x,y \rangle z^{k}
&=\sum_{k=0}^{\infty}\langle T^{2k}x,y\rangle z^{2k}+\sum_{k=0}^\infty \langle T^{2k+1}x,y\rangle z^{2k+1}\\
&=\sum_{k=0}^\infty \langle T^{2k}x,T^* y \rangle z^{2k+1}\\
&=z\sum_{i=1}^d (\sum_{k=0}^\infty (\frac{z^2}{\lambda_i})^k)\\
&=z(\sum_{i=1}^d \frac{1}{1-\frac{z^2}{\lambda_i}}).
\end{align*}
By Theorem \ref{T:variant of C}, $T$ is intransitive.
\end{proof}


\bibliographystyle{amsplain}

\end{document}